\documentclass[12pt]{amsart}

\title[Counting hyperbolic manifolds and a bound on torsion]{Counting non-commensurable hyperbolic manifolds and a bound on homological torsion}
\author{Bram Petri}
\date{\today}

\address{Mathematisches Institut der Universit\"at Bonn, Germany}

\usepackage{amsmath,amsfonts,amssymb,amsthm,graphicx,overpic,hyperref}
\usepackage{float,enumitem}
\usepackage[toc,page]{appendix}
\usepackage[T1]{fontenc}

\newtheorem{thmnew}{Theorem}
\newtheorem{cornew}{Corollary}
\newtheorem{thm}{Theorem}[section]
\newtheorem{prp}[thm]{Proposition}

\newtheorem{lem}[thm]{Lemma}
\newtheorem{con}[thm]{Conjecture}

\newtheorem{innerthmrep}{Theorem}
\newenvironment{thmrep}[1]
  {\renewcommand\theinnerthmrep{#1}\innerthmrep}
  {\endinnerthmrep}

\newtheorem{innercorrep}{Corollary}
\newenvironment{correp}[1]
  {\renewcommand\theinnercorrep{#1}\innercorrep}
  {\endinnercorrep}

\theoremstyle{definition}
\newtheorem{dff}[thm]{Definition}

\usepackage[margin=1.4in]{geometry}
\newcommand{\nc}{\newcommand}
\nc{\dmo}{\DeclareMathOperator}

\nc{\abs}[1]{\left| #1 \right|}
\nc{\bigO}[1]{\mathcal{O}\left(#1\right)}
\nc{\card}[1]{\left|#1\right|}
\nc{\ceil}[1]{\left\lceil #1 \right\rceil}
\nc{\CC}{\mathbb{C}}
\nc{\floor}[1]{\left\lfloor #1 \right\rfloor}
\dmo{\Id}{Id}
\nc{\ZZ}{\mathbb{Z}}
\nc{\len}[1]{\left| #1 \right|}
\nc{\littleo}[1]{o\left(#1\right)}
\dmo{\Mat}{Mat}
\nc{\NN}{\mathbb{N}}
\nc{\norm}[1]{\left|\left| #1 \right|\right|}
\nc{\QQ}{\mathbb{Q}}
\nc{\RR}{\mathbb{R}}
\nc{\st}[2]{\left\{ #1 ;\; #2\right\}}
\dmo{\supp}{supp}
\nc{\tr}[1]{\mathrm{tr}\left(#1\right)}

\dmo{\area}{area}
\dmo{\conv}{conv}
\dmo{\diam}{diam}
\dmo{\dist}{\mathrm{d}}
\nc{\HH}{\mathbb{H}}
\dmo{\inj}{inj}
\dmo{\Isom}{Isom}
\dmo{\MCG}{MCG}
\dmo{\MPL}{MPL}
\dmo{\Mod}{\mathcal{M}}
\dmo{\PL}{PL}
\nc{\Sphere}{\mathbb{S}}
\dmo{\sys}{sys}
\dmo{\Teich}{\mathcal{T}}
\nc{\Torus}{\mathbb{T}}
\dmo{\vol}{vol}
\dmo{\WP}{WP}

\dmo{\convTV}{\;\stackrel{\mathrm{TV}}{\longrightarrow}\;}
\nc{\ExV}[2]{\mathbb{E}_{#1}\left[#2\right]}
\dmo{\EE}{\mathbb{E}}
\nc{\Pro}[2]{\mathbb{P}_{#1}\left[#2\right]}
\dmo{\PP}{\mathbb{P}}
\nc{\distTV}[2]{\mathrm{d}_{\rm TV}\left(#1,#2\right)}
\dmo{\UU}{\mathbb{U}}
\nc{\Var}[2]{\mathbb{V}\mathrm{ar}_{#1}\left[#2\right]}

\dmo{\alt}{\mathfrak{A}}
\dmo{\Aut}{Aut}
\dmo{\Cay}{Cay}
\dmo{\Fix}{Fix}
\dmo{\FF}{\mathbb{F}}
\dmo{\GL}{GL}
\dmo{\Hom}{Hom}
\dmo{\KK}{\mathbb{K}}
\dmo{\PSL}{PSL}
\dmo{\quat}{\mathcal{H}}
\dmo{\Rep}{Rep}
\dmo{\sym}{\mathfrak{S}}
\dmo{\SL}{SL}

\begin{document}

\begin{abstract} 
We prove that the cardinality of  the torsion subgroups in homology of a closed hyperbolic manifold of any dimension can be bounded by a doubly exponential function of its diameter. It would follow from a conjecture by Bergeron and Venkatesh that the order of growth in our bound is sharp. 

We also determine how the number of non-commensurable closed hyperbolic manifolds of dimension at least $3$ and bounded diameter grows. The lower bound implies that the fraction of arithmetic manifolds tends to zero as the diameter goes up.
\end{abstract}

\maketitle

\section{Introduction}

Recently there has been a lot of progress on understanding the relation between the volume of a negatively curved manifold and its toplogical complexity. In this note, we will instead consider the relation between complexity and diameter. We will restrict to closed hyperbolic  (constant sectional curvature $-1$) manifolds. The main upshot of considering the diameter instead of the volume is that we obtain bounds in dimension $3$.

\subsection{New results}
Recall that two manifolds are called commensurable if they have a common finite cover. The diameter of a commensurability class of manifolds is the minimal diameter realized by a manifold in that class. Given $d\in\RR_+$ and $n\geq 3$, let $NC^{\diam}_n(d)$ denote the number of commensurability classes of closed hyperbolic $n$-manifolds of diameter $\leq d$.  We will prove
\begin{thmnew}\label{thm_diam} For all $n \geq 3$ there exist $0<a<b$ so that
\[a\cdot d \leq \log(\log(NC^{\diam}_n(d))) \leq b \cdot d\]
for all $d\in\RR_+$ large enough.
\end{thmnew}
Note that the analogous statement in dimension $2$ is false. Since surfaces have large deformation spaces of hyperbolic metrics (see eg. \cite{Bus} for details), it is not hard to produce an uncountable number of non-commensurable hyperbolic surfaces that have both bounded diameter and bounded volume.

Our upper bound follows directly from a result by Young \cite{You} (see also Equation \eqref{eq_diamcount}) that estimates the number of manifolds up to a given diameter. However, for his lower bound, Young uses finite covers of a fixed manifold, which are all commensurable. We will instead consider a collection of non-commensurable manifolds constructed by Gelander and Levit \cite{GelLev} and will use results on random graphs due to Bollob\'as and Fernandez de la Vega \cite{BolFer} to argue that most of these manifolds have small diameters. 

Because $NC^{\diam}_n(d)$ is finite for all $n\geq 3$ and $d\in \RR_+$, we can turn the set of commensurability classes of closed hyperbolic $n$-manifolds of diameter $\leq d$ into a probability space by equipping it with the uniform probability measure. That is, given $n\geq 3$ and $d\in \RR_+$ and a set $A$ of commensurability classes of closed hyperbolic manifolds of diameter $\leq d$, we set
\[\PP_{n,d}[A] = \frac{\card{A}}{NC^{\diam}_n(d)},\]
where $\card{A}$ denotes the cardinality of $A$. 

It follows from Theorem \ref{thm_diam} together with results by Belolipetsky \cite{Bel} for $n\geq 4$ and Belolipetsky, Gelander, Lubotzky and Shalev \cite{BelGelLubSha} for $n=3$ that most maximal lattices are not arithmetic:

\begin{cornew}\label{cor_arithm} Let $n\geq 3$. We have
\[\lim_{d\to\infty}\PP_{n,d}[\text{The manifold is arithmetic}] = 0.\]
\end{cornew}
Similar results have been proved in dimension $\geq 4$ by Gelander and Levit \cite{GelLev} with diameter replaced by volume and by Masai \cite{Mas} for a different model of random $3$-manifolds: random $3$-dimensional mapping tori built out of punctured surfaces.

We show that for closed hyperbolic manifolds, the size of homological torsion can also be bounded in terms of the diameter of the manifold.
\begin{thmnew}\label{thm_torsion} For every $n\geq 2$ there exists a constant $C >0$ so that 
\[\log\log\left(\card{H_i(M,\ZZ)_{\mathrm{tors}}}\right) \leq C\cdot \diam(M) \]
for all $i=0,\ldots,n$ for any closed hyperbolic $n$-manifold $M$.
\end{thmnew}

Let us first note that for $n=2$ our theorem is automatic, since all torsion subgroups in the homology of a closed hyperbolic surface are trivial. Moreover, in dimension at least $4$ results by Bader, Gelander and Sauer \cite{BadGelSau} (see also Equation \eqref{eq_voltors}) together with a comparison between volume and diameter (Lemma \ref{lem_diamvol}) also prove Theorem \ref{thm_torsion}. However, it is known that in dimension $3$, this method cannot work.

We also note that to prove that our theorem is sharp, the most likely examples would be sequences hyperbolic manifolds with exponentially growing torsion and logarithmically growing diameter. Examples of such sequences are abelian covers \cite{SilWil,Rai1,Fellows}, Liu's construction \cite{Liu} and random Heegaard splittings \cite{Kow,Fellows}. However, these manifolds are not known to have small diameters and in some cases are even known not to have small diameters. On the other hand, certain sequences of covers of arithmetic manifolds are conjectured to have exponential torsion growth by Bergeron and Venkatesh \cite{BerVen}. It is known that the corresponding lattice has Property $(\tau)$ with respect to this sequence of covers \cite{SarXue}, from which it follows that their diameter is small by a result of Brooks \cite{Bro} (see Sections \ref{sec_diambds} and \ref{sec_arithm}). So assuming the conjecture by Bergeron and Venkatesh, these manifolds would saturate the bound in Theorem \ref{thm_torsion} up to a multiplicative constant.

The proof of Theorem \ref{thm_torsion} consists of two steps. First we use Young's method from \cite{You} to build a simplicial complex that models our manifold and has a bounded number of cells of any dimension. We then use a lemma due to Bader, Gelander and Sauer \cite{BadGelSau} (based on a lemma of Gabber) that bounds the homological torsion in terms of the number of cells (Lemma \ref{lem_tors}) to derive our bound.

\subsection{Volume and complexity}
The main motivation for our work is formed by results that bound the homological complexity of a complete negatively curved $n$-manifold $M$ in terms of its volume. 

A classical result due to Gromov and worked out by Ballmann, Gromov and Schr\"oder \cite{Gro2,BalGroSch}, states that there exists a constant $C >0$, depending only on $n$, so that the Betti numbers $b_i(M)$ satisfy
\begin{equation}\label{eq_volbetti}
 b_i(M) \leq C \cdot \vol(M),
\end{equation}
for $i=0,\ldots,n$, where $\vol(M)$ denotes the volume of $M$.

More recently, Bader, Gelander and Sauer \cite{BadGelSau} have shown that, when the dimension $n$ is at least $4$, the cardinality of the torsion subgroups $H_i(M,\ZZ)_{\mathrm{tors}}$ in homology can also be bounded in terms of the volume. They show that for all $n\geq 4$, there exists a constant $C>0$, depending only on $n$, so that 
\begin{equation}\label{eq_voltors}
 \log\left(\card{H_i(M,\ZZ)_{\mathrm{tors}}}\right) \leq C\cdot \vol(M).
\end{equation}

\subsection{Counting manifolds by volume}
We will again restrict ourselves to closed hyperbolic manifolds. A classical result due to Wang \cite{Wan} states that in dimension $n\geq 4$ the number of closed hyperbolic manifolds of volume $\leq v$ is finite for any $v\in \RR_+$. Let $N^{\vol}_n(v)$ denote the number such manifolds. The first bounds on the growth of $N^{\vol}_n(v)$ are due to Gromov \cite{Gro1}. Burger, Gelander, Lubotzky and Mozes \cite{BurGelLubMoz} showed that for all $n\geq 4$ there exist $0<a<b \in \RR$ such that 
\begin{equation}\label{eq_volcount}
a \cdot v\log(v)\leq \log\left(N^{\vol}_n(v)\right)\leq b\cdot v\log(v),
\end{equation}
for all $v\in \RR_+$ large enough.

Analogously to the case of the diameter, the volume of a commensurability class of manifolds is the minimal volume realized in that class. Let $NC^{vol}_n(v)$ denote the number of commensurability classes of hyperbolic $n$-manifolds of volume $\leq v$. The first lower bounds on this number are due to Raimbault \cite{Rai2}. Gelander and Levit \cite{GelLev} showed that for all $n\geq 4$ there exist $0<a<b \in \RR$ such that 
\begin{equation}\label{eq_volcommcount}
a\cdot v\log(v)\leq \log\left(NC^{\vol}_n(v)\right)\leq b\cdot v\log(v),
\end{equation}
for all $v\in \RR_+$ large enough. 

\subsection{$3$-dimensional manifolds}

As opposed to in dimension $4$ and above, the number of closed hyperbolic $3$-manifolds of bounded volume is not finite. This for instance follows from work of Thurston (see for example \cite[Chapter E]{BenPet}). This means that there is also no reason to suppose that a statement like Equation \eqref{eq_voltors} holds in dimension $3$. In fact, in \cite{BadGelSau}, the authors prove that no such bound can exist, even for sequences of hyperbolic $3$-manifolds manifolds that Benjamini-Schramm converge to $\HH^3$. We note that Fr\k{a}czyk \cite{Fra} has however proved that for arithmetic manifolds, a similar bound to that of Bader, Gelander and Sauer does hold.

\subsection{Diameters}
The number of closed hyperbolic $n$-manifolds of diameter at most $d$ is finite for any $n\geq 3$ and $d\in \RR_+$. The best known estimates on the number $N^{\diam}_n(d)$ of closed hyperbolic $n$-manifolds of diameter $\leq d$ are due to Young \cite{You}. He proved that for every $n\geq 3$ there exist constants $0<a<b\in\RR^+$ so that 
\begin{equation}\label{eq_diamcount}
a\cdot d \leq \log(\log(N^{\diam}_n(d)))\leq b\cdot d
\end{equation}
for all $d\in \RR_+$ large enough. 

Equation \eqref{eq_volbetti} implies that the Betti numbers of a closed hyperbolic manifold $M$ can also be bounded by its diameter $\diam(M)$ as follows
\[
b_i(M) \leq C \cdot e^{(n-1)\cdot \diam(M)},
\]
for all $i=0,\ldots n$, where $C>0$ is a constant depending only on $n$ (see Lemma \ref{lem_diamvol}). Moreover, because there are hyperbolic surfaces with linearly growing genus and logarithmically growing diameter (for instance random surfaces \cite{BroMak,Mir}), a bound of this generality is necessarily exponential in diameter.

\subsection*{Acknowledgement} The author's research was supported by the ERC Advanced Grant ``Moduli''. The author also thanks the organizers of the Borel Seminar 2017, during which the research for this article was done. Finally, the author thanks Filippo Cerocchi, Jean Raimbault and Roman Sauer for useful conversations.

\section{Background material}

In what follows, $n$ will be a natural number, $M$ a closed oriented hyperbolic $n$-manifold. We will use $\vol(M)$, $\diam(M)$, $\inj(M)$ and $\lambda_1(M)$ to denote the volume, the diameter, the injectivity radius and the first non-zero eigenvalue of the Laplace-Beltrami operator of $M$ respectively. Moreover, $\dist:M\times M\to \RR_+$ will denote the distance function on $M$. Finally, $\HH^n$ will denote hyperbolic $n$-space and $\Isom^+(\HH^n)$ will denote its group of orientation preserving isometries.

\subsection{Bounds involving the diameter}\label{sec_diambds}
Many of our bounds are based on the following well known fact. 
\begin{lem}\label{lem_diamvol} Let $n\geq 2$. There exists a constant $C>0$, depending only on $n$, so that for every closed hyperbolic $n$-manifold $M$ we have
\[C\cdot \log(\vol(M)) \leq \diam(M).  \]
\end{lem}
\begin{proof} The crucial observation is that the ball $B_M(p,\diam(M))$ of radius $\diam(M)$ around any point $p\in M$, by definition of the diameter, covers $M$. This implies that
\[\vol(M)\leq \vol(B_{M}(p,\diam(M))).\]
On the other hand, the volume of $B_{M}(p,\diam(M))$ is at most the volume of a ball of the same radius in $\HH^n$. The volume of a ball $B_{\HH^n}(p,R)$ of radius $R$ around $p\in\HH^n$ is equal to 
\[\vol(B_{\HH^n}(p,R))=\vol(\Sphere^{n-1})\int_0^R \sinh^{n-1}(t)dt,\] 
where $\Sphere^{n-1}$ denotes the $(n-1)$-sphere equipped with the round metric (see for instance \cite[\S 3.4]{Rat}). Putting this together with the inequality above gives the lemma.
\end{proof}

Moreover, we shall need a bound on the injectivity radius in terms of the diameter. The following was proved by Young \cite{You}, based on work by Reznikov \cite{Rez}:

\begin{lem}\label{lem_diaminj} Let $n\geq 2$. There exists a constant $C>0$, depending only on $n$, so that
\[\inj(M) \geq \exp(-\diam(M)/C)\]
for all closed hyperbolic $n$-manifolds $M$.
\end{lem}

In order to control the diameter of a sequence of congruence covers later on, we will use the eigenvalue of their Laplacian in combination with the following theorem due to Brooks \cite[Theorem 1]{Bro}:

\begin{thm}\label{thm_diamlaplace} Let $M$ be a closed hyperbolic manifold and Let $\{M_i\}_{i\in\mathcal{I}}$ be a family of finite covers of $M$. If there exists a constant $C>0$ so that $\lambda_1(M_i)> C$ for all $i\in\mathcal{I}$, then there exist constants $a,b,c>0$ such that
\[a < \frac{\log(\vol(M_i))+c}{\diam(M_i)} < b\]
for all $i\in\mathcal{I}$.
\end{thm}

\subsection{Gelander and Levit's construction}
The lower bound in Theorem \ref{thm_diam} will come from a construction due to Gelander and Levit, which is inspired by a classical construction due to Gromov and Piatetski-Shapiro \cite{GroPia} (see also \cite{Rai2}). We will briefly describe some, but not all, of the details of their construction. For more information we refer to \cite{GelLev}.

Assume we are given six compact hyperbolic $n$-manifolds with boundary $V_0$, $V_1$, $A_+$, $A_-$, $B_+$ and $B_-$ so that
\begin{itemize}
\item[-] $V_0$ and $V_1$ both have four boundary components and $A_+$, $A_-$, $B_+$ and $B_-$ all have two boundary components.
\item[-] All the boundary components of these manifolds are isometric to a fixed closed hyperbolic $(n-1)$-manifold.
\item[-] Each of these six manifolds is embedded in an arithmetic manifolds without boundary that are pairwise non-commensurable (see Section \ref{sec_arithm} for a definition of an arithmetic group).
\end{itemize}
In \cite[Section 4]{GelLev}, Gelander and Levit explain how to construct these manifolds.

We will glue these manifolds according to Schreier graphs for finite index subgroups of the free group $\FF_2=\langle a,b\rangle$. Let $\Cay(\FF_2,\{a,b\})$ denote the Cayley graph of $\FF_2$ with respect to the generating set $\{a,b\}$. Given $H < \FF_2$, the Schreier graph $\Gamma_H$ is the graph
\[\Gamma_H = \Cay(\FF_2,\{a,b\}) / H.\]
Since the edges in $\Cay(\FF_2,\{a,b\})$ come with a natural labeling with the symbols $\{a^\pm,b^\pm\}$, the edges $\Gamma_H$ come with such a labeling as well. Furthermore, note that the number of vertices of $\Gamma_H$ is equal to the index $[\FF_2:H]$. Let us denote the vertex and edge set of $\Gamma_H$ by $V(\Gamma_H)$ and $E(\Gamma_H)$ respectively. 

\begin{dff} Given a finite index subgroup $H < \FF_2$ and a map $\tau:V(\Gamma_H)\to\{0,1\}$, we construct the closed hyperbolic $n$-manifold $M(H,\tau)$ as follows:
\begin{itemize}
\item[-] To each vertex $v\in V(\Gamma_H)$, associate a copy of $V_{\tau(v)}$
\item[-] and to each edge $e\in E(\Gamma_H)$, associate a copy of the pair $A^+,A^-$ or $B^+,B^-$, according to whether it is labeled with an $a^\pm$ or a $b^\pm$.
\item[-] Glue the manifolds together according to the incidence relations in $\Gamma_H$. In particular, the order in which to glue the two blocks associated to an edge depends on whether or not the edge is labeled with an inverse.
\end{itemize}
\end{dff}

Note that there is some ambiguity in the construction above: there is for instance a choice which boundary component to glue to which. Since we are using the construction for a lower bound, this won't make a difference to us. We will from now on assume some choice of gluing is given for every pair $(H,\tau)$. Figure \ref{pic_construction} shows a cartoon of what the local picture of $M(H,\tau)$ might look like:

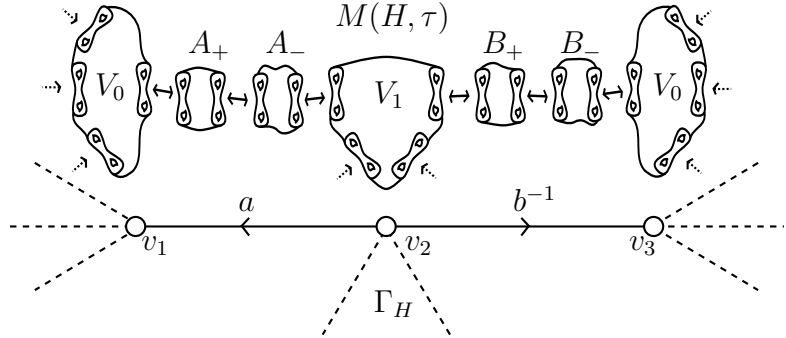
\begin{figure}[H]
\begin{center}
\begin{overpic}[scale=1]{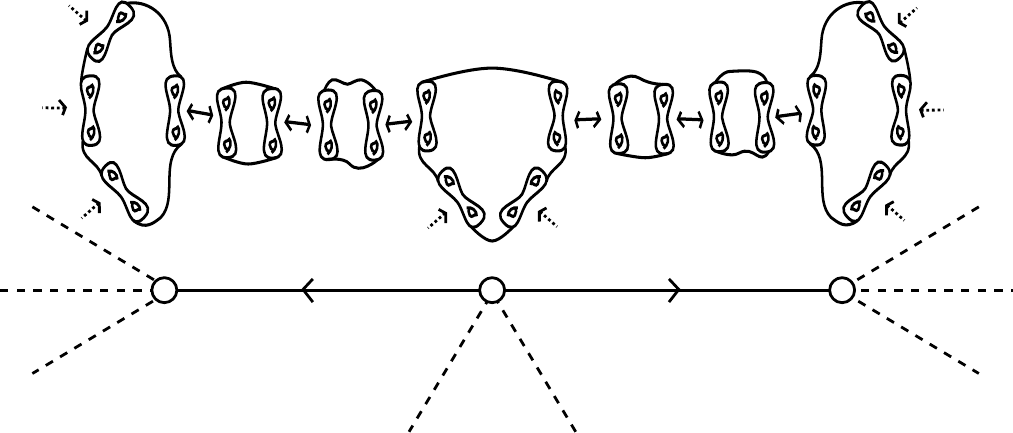}
\put (11,31) {$V_0$}
\put (23,36.5) {$A_+$}
\put (33,36.5) {$A_-$}
\put (47,30) {$V_1$}
\put (61,36.5) {$B_+$}
\put (71,36.5) {$B_-$}
\put (83,31) {$V_0$}
\put (29.5,16) {$a$}
\put (65,16) {$b^{-1}$}
\put (47,3) {$\Gamma_H$}
\put (42,40) {$M(H,\tau)$}
\put (17,11) {$v_1$}
\put (51,11) {$v_2$}
\put (80,11) {$v_3$}
\end{overpic}
\end{center}
\caption{A local picture of $\Gamma_H$ and $M(H,\tau)$, where $\tau(v_1)=\tau(v_3)=0$ and $\tau(v_2)=1$}
\label{pic_construction}
\end{figure}

In \cite[Proposition 3.3]{GelLev}, Gelander and Levit show:
\begin{prp}\label{prp_noncomm} Let $H,H'<\FF_2$ be distinct finite index subgroups and let $\tau:V(\Gamma_H)\to \{0,1\}$ and $\tau':V(\Gamma_{H'})\to \{0,1\}$ be so that 
\[\card{\tau^{-1}(1)} = \card{(\tau')^{-1}(1)} = 1.\]
Then $M(H,\tau)$ and $M(H',\tau')$ are not commensurable.
\end{prp}

The upshot of this proposition is that the construction of Gelander and Levit gives rise to at least $a_N(\FF_2)$ non-commensurable manifolds on built out of graphs with $n$ vertices, where $a_N(\FF_2)$ denotes the number of index $N$ subgroups of $\FF_2$.

\subsection{Graphs and groups}

To work with Gelander and Levit's construction, we will need two bounds. We need a lower bound on $a_N(\FF_2)$ and we need to know what the typical diameter of a Schreier graph of an index $N$ subgroup of $\FF_2$ is.

For details on the subgroup growth of $\FF_2$, we refer to Chapter 2 in the monograph by Lubotzky and Segal \cite{LubSeg}. We will use the following bound, that can be found as a special case of \cite[Theorem 2.1]{LubSeg}:
\begin{thm}\label{thm_subgrps} We have
\[ a_N(\FF_2)\sim N \cdot N!\]
as $N\to\infty$.
\end{thm}
In the theorem above, we write $f(N)\sim g(N)$ as $N\to\infty$ to mean that 
\[\lim_{N\to\infty} f(N)/g(N) =1.\]

The distance between two vertices in a connected graph is the minimal number of edges in a path between these two vertices. The diameter $\diam(\Gamma)$ of a finite graph $\Gamma$ is the maximal distance realized by two vertices in $\Gamma$.

To control the diameter of a typical Schreier graph we will use results from random graphs. The fact that the diameter of a random regular graph is bounded by a logarithmic function of the number of vertices can for instance be derived from the fact that a random regular graph has a large spectral gap with probability tending to $1$ as the number of vertices tends to infinity (see \cite{Fri,Pud,LinHooWig,BroSha}). The sharpest result however does not use this method and is due to Bollob\'as and Fernandez de la Vega \cite[Theorem 3]{BolFer}. We will state their result only in the case of $4$-regular graphs. 
\begin{thm}\label{thm_diamgraph} Let $\PP_N$ denote the uniform probability measure on the set of isomorphism classes of $4$-regular graphs on $N$. There exists a function $E:\NN\to\RR$ so that
\[E(N) = o(\log(N))\]
as $N\to\infty$ and
\[\PP_N[\text{The graph has diameter }\leq \log_3(N) + E(N)]\to 1\]
as $N\to\infty$.
\end{thm}
We note that the bound is basically as low as one could possibly expect. Indeed, with an argument very similar to that in Lemma \ref{lem_diamvol} it can be shown that the diameter of a $4$-regular graph is at least of the order $\log_3(N)$. 

We won't go into random regular graphs in this note and refer the reader to \cite{BolFer,Bol,Wor} for the details. We do however note that uniformly picking an index $N$ Schreier graph is a slightly different model for random $4$-valent graphs than the uniform probability measure on isomorphism classes of $4$-valent graphs. It turns out that the two models are what is called contiguous: they have the same asymptotic $0$-sets, which is enough for our purposes. More details on this can be found in \cite[Section 4]{Wor} and \cite{GreJanKimWor}.

\subsection{Arithmetic manifolds}\label{sec_arithm}
Let $G$ be a semisimple Lie group of noncompact type that is defined over $\QQ$ (in our case $G$ will always be $\Isom^+(\HH^n)$). A discrete subgroup $\Gamma < G(\QQ)$ will be called arithmetic if there is a $\QQ$-embedding $\rho:G\to \GL_m(\RR)$ such that $\rho(\Gamma)$ is commensurable with $G(\ZZ) = \GL_m(\ZZ)\cap \rho(G)$. Arithmetic groups come with a sequence of finite index subgroups called congruence groups. For general background on arithmetic groups, we refer to \cite{MacRei,Mor}. 

Recall that a lattice $\Gamma < \Isom^+(\HH^n)$ is called uniform if $\Gamma\backslash\Isom^+(\HH^n)$ is compact. We will call $\Gamma$ maximal if it is not properly contained in another lattice. The result we will need is a bound on the number of maximal uniform arithmetic lattices up to a given covolume in $\Isom^+(\HH^n)$ for $n\geq 3$. To this end, let $\mathrm{MAL}^u_n(v)$ denote the number of maximal uniform arithmetic lattices of covolume $\leq v$ in $\Isom^+(\HH^n)$. The following theorem is due to Belolipetsky \cite{Bel} in dimension $n\geq 4$ and Belolipetsky, Gelander, Lubotzky and Shalev \cite{BelGelLubSha} in dimensions $2$ and $3$.
\begin{thm}\label{thm_mal} Let $n\geq 2$ and $\varepsilon >0$. There exist constants $\alpha=\alpha(n) \in \RR_+$ and $\beta=\beta(n,\varepsilon)\in\RR_+$ so that 
\[v^\alpha \leq \mathrm{MAL}^u_n(v) \leq v^{\beta\cdot (\log v)^\varepsilon}\]
for all $v\in\RR^+$ large enough.
\end{thm}

Let us now restrict to hyperbolic $3$-manifolds. To get bounds on the diameters of congruence covers of arithmetic manifolds, we will use the following theorem due to Sarnak and Xue \cite{SarXue}:
\begin{thm}\label{thm_tau} Let $\Gamma<\Isom^+(\HH^3)$ be a uniform arithmetic lattice. Then there exists a constant $C=C(\Gamma)>0$ so that for all congruence subgroups $\Gamma'<\Gamma$ we have
\[\lambda_1(\Gamma'\backslash \HH^3) \geq C.\]
\end{thm}

If a lattice $\Gamma<\SL_2(\CC)$ has a sequence $\{\Gamma_N\}_{N}$ of finite index subgroups so that $\lambda_1(\Gamma_N\backslash \HH^3)$ is uniformly bounded from below for all $N\in\NN$, then $\Gamma$ is said to have Property ($\tau$) with respect to this sequence. In the special case where $\Gamma$ is arithmetic and the sequence consists of congruence subgroups, Property ($\tau$) is sometimes also called the Selberg property.  

Congruence covers are also believed to have large torsion subgroups in their homology. Specifically, there is the following conjecture, due to Bergeron and Venkatesh \cite{BerVen}, which we state in the special case of hyperbolic $3$-manifolds:
\begin{con}\label{con_berven} Let $\Gamma<\SL_2(\CC)$ be a uniform arithmetic lattice and $\ldots<\Gamma_N<\Gamma_{N-1}<\ldots<\Gamma_1<\Gamma$ a sequence of congruence subgroups of $\Gamma$ so that $\cap_N \Gamma_N = \{1\}$. Then
\[\lim_{N\to\infty} \frac{\log\left(\card{H_1(\Gamma_N \backslash \HH^3,\ZZ)_\mathrm{tors}}\right)}{\vol(\Gamma_N \backslash \HH^3)} = \frac{1}{6\pi}.\]
\end{con}
The reason for the constant $1/6\pi$ above is that it is the $\ell^2$-torsion of $\HH^3$.

\subsection{Torsion and the nerve lemma}

To bound torsion in our manifolds we will use a lemma by Bader, Gelander and Sauer \cite[Lemma 5.2]{BadGelSau}, that they derived from a lemma due to Gabber (which can for instance be found in \cite[Lemma 1]{Sou}). In this lemma, the degree of a vertex ($0$-cell) in a simplicial complex is the degree of that vertex in the $1$-skeleton of the given complex.

\begin{lem}\label{lem_tors} For all $D,p \in \NN$ there exists a constant $C=C(D,p)>0$ so that for any simplicial complex $X$ with $\leq V$ vertices that all have degree $\leq D$ we have:
\[ \log\left(\card{H_p(X,\ZZ)_{\mathrm{tors}}}\right) \leq C\cdot V.\]
\end{lem}

The simplicial complex we will use will be the nerve of an open cover of our manifold. Recall that an open cover of a space $X$ is a collection $\mathcal{U}=\{U_i\}_{i\in\mathcal{I}}$ of open subsets of $X$ so that 
\[ X = \bigcup_{i\in\mathcal{I}}U_i.\]
The nerve $\mathcal{N}(\mathcal{U})$ of this cover is the simplicial complex that has the sets $U_i$ as vertices and contains a $k$-simplex for every $k$-tuple of elements in $\mathcal{U}$ that have a non-trivial intersection. See \cite[Section 4G]{Hat} for more details.

The following statement is known as the nerve lemma and can for instance be found as \cite[Corollary 4G.3]{Hat}:
\begin{lem}\label{lem_nerve}
If $\mathcal{U}$ is an open cover of a paracompact space $X$ such that every
nonempty intersection of finitely many sets in $\mathcal{U}$ is contractible, then $X$ is homotopy equivalent to the nerve $\mathcal{N}(\mathcal{U})$.
\end{lem}

\section{Counting}

We are now ready to prove Theorem \ref{thm_diam}:
\begin{thmrep}{\ref{thm_diam}} For all $n \geq 3$ there exist $0<a<b$ so that
\[a\cdot d \leq \log(\log(NC^{\diam}_n(d))) \leq b \cdot d\]
for all $d\in\RR_+$ large enough.
\end{thmrep}
\begin{proof} The upper bound is direct from Young's result (Equation \eqref{eq_diamcount}), so we focus on the lower bound. 

Consider the building blocks defined by Gelander and Levit and set
\[D=\max\{\diam(V_0),\diam(V_1),\diam(A_+),\diam(A_-),\diam(B_+),\diam(B_-)\}. \]
Because all the building blocks are compact, this is a finite number. Given a finite index subgroup $H<\FF_2$ and a map $\tau:V(\Gamma_H)\to\{0,1\}$, we have
\[\diam(M(H,\tau)) \leq 2D\cdot\diam(\Gamma_H)+2D.\]
Indeed, suppose $x,y\in M(H,\tau)$. Then the number of building blocks that need to be crossed to get from $x$ to $y$ is at most $2\diam(\Gamma_H)+2$, including the building blocks containing $x$ and $y$.

Because of Theorem \ref{thm_subgrps} combined with Theorem \ref{thm_diamgraph} and Stirling's approximation, there exists a $C>0$ so that the index $N$ subgroups of $\FF_2$ for $N$ large enough produce at least 
\[C\cdot N^{CN}\]
non-isomorphic graphs of diameter $\leq \log_3(N) + o(\log(N))$. If we now define maps $\tau:V(\Gamma)\to\{0,1\}$ that assign the value $1$ to only one vertex per graph $\Gamma$, we obtain $C\cdot N^{CN}$ manifolds of diameter $d \leq 2D\log_3(N) + o(\log(N))$. Working this out, we see that this number of manifolds is at least
\[\exp(C'\cdot d \cdot \exp(C'\cdot d)),\]
for some $C'>0$. Proposition \ref{prp_noncomm} tells us that none of the resulting manifolds will be commensurable.
\end{proof}

Corollary \ref{cor_arithm} now also easily follows.

\begin{correp}{\ref{cor_arithm}} Let $n\geq 3$. We have
\[\lim_{d\to\infty}\PP_{n,d}[\text{The manifold is arithmetic}] = 0\]
\end{correp}

\begin{proof} The only thing we need to control is the number of maximal uniform arithmetic lattices of diameter $\leq d$ in $\Isom^+(\HH^n)$. Let us call this number $\mathrm{MALD}^u_{n}(d)$. By Lemma \ref{lem_diamvol} we have
\[\mathrm{MALD}^u_{n}(d) \leq \mathrm{MAL}^u_n(e^{d/C_n})\]
for some $C_n>0$ independent of $d$. As such, Theorem \ref{thm_mal} implies that for every $\varepsilon>0$ there exists a $\beta'>0$ so that
\[\mathrm{MALD}^u_{n}(d) \leq \exp(\beta'\cdot d^{1+\varepsilon}).\]
Comparing this to Theorem \ref{thm_diam} gives the result.
\end{proof}

\section{Torsion}

In this section we prove Theorem \ref{thm_torsion} and explain how a positive answer to Conjecture \ref{con_berven} would lead to a sequence of manifolds that saturates the bound in that theorem up to a multiplicative constant.

\subsection{An upper bound for torsion in homology} We will start with:
\begin{thmrep}{\ref{thm_torsion}} For every $n\geq 2$ there exists a constant $C >0$ so that 
\[\log\log\left(\card{H_i(M,\ZZ)_{\mathrm{tors}}}\right) \leq C\cdot \diam(M) \]
for all $i=0,\ldots,n$ for any closed hyperbolic $n$-manifold $M$.
\end{thmrep}

\begin{proof} Our final goal is to employ Lemma \ref{lem_tors}. In the language of 
\cite{BadGelSau}: we need to show that a closed hyperbolic manifold $M$ is homotopy equivalent to a $(D,C\cdot\diam(M))$-simplicial complex, where $C,D>0$ are constants depending only its dimension. The simplicial complex we build is the same as that used for the upper bound in Young's result (Equation \eqref{eq_diamcount}). 

Set $r=\inj(M)$ and let $S\subset M$ be a maximal set of points so that
\[\dist(s,s') \geq r/4\]
for all $s,s'\in S$. Now consider the collection \[\mathcal{U}=\{B_M(s,r/2)\}_{s\in S},\]
where $B_M(s,r/2)$ denotes the open ball in $M$ of radius $r/2$ around $s$. It follows from maximality of $S$ that these balls form an open cover. Moreover, because their radius is half the injectivity radius they are isometrically embedded $n$-dimensional hyperbolic balls. As such they are convex, which means that their intersections are convex and thus contractible. Hence, the nerve lemma (Lemma \ref{lem_nerve}) applies. This means that the homology groups we are after are those of $\mathcal{N}(\mathcal{U})$.

So we need to find bounds on the number of vertices and their degrees in $\mathcal{N}(\mathcal{U})$ in order to apply Lemma \ref{lem_tors}.

The number of vertices, or equivalently the number of points in $S$, can be bounded by
\[\card{S} \leq \frac{\vol(M)}{\vol(B_{\HH^n}(p,r/4))} \leq D\cdot \frac{\vol(M)}{(r/4)^n}, \]
where $B_{\HH^n}(p,r/4)$ is the ball of radius $r/4$ around some point $p\in\HH^n$ and $D>0$ is some constant depending only on the dimension. The second of these bounds again follows from the closed formula for the volume of a ball in $\HH^n$. Now we use Lemma \ref{lem_diaminj} and Lemma \ref{lem_diamvol}  to tell us that 
\[(r/4)^n \geq \exp(-C\diam(M)) \;\;\text{and}\;\; \vol(M) \leq \exp(C\diam(M))\]
for some $C>0$ depending only on $n$. So we obtain
\[\card{S} \leq A\cdot \exp(B\diam(M)).  \]
for some $A,B>0$ depending only on $n$.

All that remains is to show that each vertex has a bounded number of neighbors. First of all note that all the neighbors of a point $s\in S$ lie in $B_M(s,r)\subset M$. By definition of $S$, the balls of radius $r/8$ around the neighbors of $s$ are all disjoint and all lie in $B_{M}(s,9r/8)$. This means that the number of neighbors is at most
\[\frac{\vol(B_M(s,9r/8))}{\vol(B_M(s,r/8))} \leq \frac{\vol(B_{\HH^n}(s,9r/8))}{\vol(B_{\HH^n}(s,r/8))},\]
which, for $r$ small enough, is uniformly bounded in each fixed dimension.
\end{proof}

\subsection{Sharpness of the bound}

Like we said in the introduction, if Conjecture \ref{con_berven} holds then we would get a sequence of closed hyperbolic $3$-manifolds $\{M_N\}_{N\in\NN}$ such that $\diam(M_N)\to \infty$ as $N\to\infty$ and
\[ \log\log\left(\card{ H_1(M_N,\ZZ)_{\mathrm{tors}}}\right) \geq C \diam(M_N).\]
for some $C>0$ independent of $N$. Indeed, it follows from Theorem \ref{thm_diamlaplace} together with Theorem \ref{thm_tau} that for a uniform arithmetic group $\Gamma < \SL_2(\CC)$ and a sequence of congruence subgroups $\{\Gamma_N\}_N$, the manifolds $M_N = \Gamma_N \backslash \HH^3$ satisfy
\[\diam(M_N) \leq A \cdot \log(\vol(M_N)).\]
It would follow from Conjecture \ref{con_berven} that 
\[\vol(M_N)\leq B\cdot \log\left(\card{ H_1(M_N,\ZZ)_{\mathrm{tors}}}\right) \]
for all $N$ and some $B>0$ independent of $N$. Putting these two together would give the desired result.



\bibliographystyle{alpha}
\bibliography{TorsDiam.bib}

\newcommand{\etalchar}[1]{$^{#1}$}
\begin{thebibliography}{GJKW02}

\bibitem[BBG{\etalchar{+}}17]{Fellows}
Hyungryul Baik, David Bauer, Ilya Gekhtman, Ursula Hamenst\"adt, Sebastian
  Hensel, Thorben Kastenholz, and Daniel Valenzuela.
\newblock Exponential torsion growth for random 3-manifolds.
\newblock {\em Int. Math. Res. Not. IMRN}, 2017.
\newblock to appear.

\bibitem[Bel07]{Bel}
Mikhail Belolipetsky.
\newblock Counting maximal arithmetic subgroups.
\newblock {\em Duke Math. J.}, 140(1):1--33, 2007.
\newblock With an appendix by Jordan Ellenberg and Akshay Venkatesh.

\bibitem[BFdlV82]{BolFer}
B.~Bollob\'as and W.~Fernandez de~la Vega.
\newblock The diameter of random regular graphs.
\newblock {\em Combinatorica}, 2(2):125--134, 1982.

\bibitem[BGLM02]{BurGelLubMoz}
M.~Burger, T.~Gelander, A.~Lubotzky, and S.~Mozes.
\newblock Counting hyperbolic manifolds.
\newblock {\em Geom. Funct. Anal.}, 12(6):1161--1173, 2002.

\bibitem[BGLS10]{BelGelLubSha}
Mikhail Belolipetsky, Tsachik Gelander, Alexander Lubotzky, and Aner Shalev.
\newblock Counting arithmetic lattices and surfaces.
\newblock {\em Ann. of Math. (2)}, 172(3):2197--2221, 2010.

\bibitem[BGS85]{BalGroSch}
Werner Ballmann, Mikhael Gromov, and Viktor Schroeder.
\newblock {\em Manifolds of nonpositive curvature}, volume~61 of {\em Progress
  in Mathematics}.
\newblock Birkh\"auser Boston, Inc., Boston, MA, 1985.

\bibitem[BGS16]{BadGelSau}
Uri Bader, Tsachik Gelander, and Roman Sauer.
\newblock Homology and homotopy complexity in negative curvature.
\newblock Preprint, arXiv:1612.04871, 2016.

\bibitem[BM04]{BroMak}
Robert Brooks and Eran Makover.
\newblock Random construction of {R}iemann surfaces.
\newblock {\em J. Differential Geom.}, 68(1):121--157, 2004.

\bibitem[Bol01]{Bol}
B\'ela Bollob\'as.
\newblock {\em Random graphs}, volume~73 of {\em Cambridge Studies in Advanced
  Mathematics}.
\newblock Cambridge University Press, Cambridge, second edition, 2001.

\bibitem[BP92]{BenPet}
Riccardo Benedetti and Carlo Petronio.
\newblock {\em Lectures on hyperbolic geometry}.
\newblock Universitext. Springer-Verlag, Berlin, 1992.

\bibitem[Bro88]{Bro}
Robert Brooks.
\newblock Some remarks on volume and diameter of {R}iemannian manifolds.
\newblock {\em J. Differential Geom.}, 27(1):81--86, 1988.

\bibitem[BS87]{BroSha}
A.~Broder and E.~Shamir.
\newblock On the second eigenvalue of random regular graphs.
\newblock In {\em The 28th Annual Symposium on Foundations of Computer
  Science}, pages 286--294. 1987.

\bibitem[Bus10]{Bus}
Peter Buser.
\newblock {\em Geometry and spectra of compact {R}iemann surfaces}.
\newblock Modern Birkh\"auser Classics. Birkh\"auser Boston, Inc., Boston, MA,
  2010.
\newblock Reprint of the 1992 edition.

\bibitem[BV13]{BerVen}
Nicolas Bergeron and Akshay Venkatesh.
\newblock The asymptotic growth of torsion homology for arithmetic groups.
\newblock {\em J. Inst. Math. Jussieu}, 12(2):391--447, 2013.

\bibitem[Fr{\k a}17]{Fra}
Miko\l{}aj Fr{\k a}czyk.
\newblock {\em Benjamini-Schramm convergence of locally symmetric spaces}.
\newblock PhD thesis, Universit\'e Paris-Sud, 2017.

\bibitem[Fri08]{Fri}
Joel Friedman.
\newblock A proof of {A}lon's second eigenvalue conjecture and related
  problems.
\newblock {\em Mem. Amer. Math. Soc.}, 195(910):viii+100, 2008.

\bibitem[GJKW02]{GreJanKimWor}
Catherine Greenhill, Svante Janson, Jeong~Han Kim, and Nicholas~C. Wormald.
\newblock Permutation pseudographs and contiguity.
\newblock {\em Combin. Probab. Comput.}, 11(3):273--298, 2002.

\bibitem[GL14]{GelLev}
Tsachik Gelander and Arie Levit.
\newblock Counting commensurability classes of hyperbolic manifolds.
\newblock {\em Geom. Funct. Anal.}, 24(5):1431--1447, 2014.

\bibitem[GPS88]{GroPia}
M.~Gromov and I.~Piatetski-Shapiro.
\newblock Nonarithmetic groups in {L}obachevsky spaces.
\newblock {\em Inst. Hautes \'Etudes Sci. Publ. Math.}, (66):93--103, 1988.

\bibitem[Gro81]{Gro1}
Michael Gromov.
\newblock Hyperbolic manifolds (according to {T}hurston and {J}\o rgensen).
\newblock In {\em Bourbaki {S}eminar, {V}ol. 1979/80}, volume 842 of {\em
  Lecture Notes in Math.}, pages 40--53. Springer, Berlin-New York, 1981.

\bibitem[Gro82]{Gro2}
Michael Gromov.
\newblock Volume and bounded cohomology.
\newblock {\em Inst. Hautes \'Etudes Sci. Publ. Math.}, (56):5--99 (1983),
  1982.

\bibitem[Hat02]{Hat}
Allen Hatcher.
\newblock {\em Algebraic topology}.
\newblock Cambridge University Press, Cambridge, 2002.

\bibitem[HLW06]{LinHooWig}
Shlomo Hoory, Nathan Linial, and Avi Wigderson.
\newblock Expander graphs and their applications.
\newblock {\em Bull. Amer. Math. Soc. (N.S.)}, 43(4):439--561, 2006.

\bibitem[Kow08]{Kow}
E.~Kowalski.
\newblock {\em The large sieve and its applications}, volume 175 of {\em
  Cambridge Tracts in Mathematics}.
\newblock Cambridge University Press, Cambridge, 2008.
\newblock Arithmetic geometry, random walks and discrete groups.

\bibitem[Liu17]{Liu}
Yi~Liu.
\newblock Immersing quasi-fuchsian surfaces of odd euler characteristic in
  closed hyperbolic 3-manifolds.
\newblock {\em J. Differential Geom.}, 2017.
\newblock to appear.

\bibitem[LS03]{LubSeg}
Alexander Lubotzky and Dan Segal.
\newblock {\em Subgroup growth}, volume 212 of {\em Progress in Mathematics}.
\newblock Birkh\"auser Verlag, Basel, 2003.

\bibitem[Mas14]{Mas}
Hidetoshi Masai.
\newblock Fibered commensurability and arithmeticity of random mapping tori.
\newblock Preprint, arXiv:1408.0348, 2014.

\bibitem[Mir13]{Mir}
Maryam Mirzakhani.
\newblock Growth of {W}eil-{P}etersson volumes and random hyperbolic surfaces
  of large genus.
\newblock {\em J. Differential Geom.}, 94(2):267--300, 2013.

\bibitem[Mor15]{Mor}
Dave~Witte Morris.
\newblock {\em Introduction to arithmetic groups}.
\newblock Deductive Press, [place of publication not identified], 2015.

\bibitem[MR03]{MacRei}
Colin Maclachlan and Alan~W. Reid.
\newblock {\em The arithmetic of hyperbolic 3-manifolds}, volume 219 of {\em
  Graduate Texts in Mathematics}.
\newblock Springer-Verlag, New York, 2003.

\bibitem[Pud15]{Pud}
Doron Puder.
\newblock Expansion of random graphs: new proofs, new results.
\newblock {\em Invent. Math.}, 201(3):845--908, 2015.

\bibitem[Rai12]{Rai1}
Jean Raimbault.
\newblock Exponential growth of torsion in abelian coverings.
\newblock {\em Algebr. Geom. Topol.}, 12(3):1331--1372, 2012.

\bibitem[Rai13]{Rai2}
Jean Raimbault.
\newblock A note on maximal lattice growth in {${\rm SO}(1,n)$}.
\newblock {\em Int. Math. Res. Not. IMRN}, (16):3722--3731, 2013.

\bibitem[Rat06]{Rat}
John~G. Ratcliffe.
\newblock {\em Foundations of hyperbolic manifolds}, volume 149 of {\em
  Graduate Texts in Mathematics}.
\newblock Springer, New York, second edition, 2006.

\bibitem[Rez95]{Rez}
Alexander Reznikov.
\newblock The volume and the injectivity radius of a hyperbolic manifold.
\newblock {\em Topology}, 34(2):477--479, 1995.

\bibitem[Sou99]{Sou}
C.~Soul\'e.
\newblock Perfect forms and the {V}andiver conjecture.
\newblock {\em J. Reine Angew. Math.}, 517:209--221, 1999.

\bibitem[SW02]{SilWil}
Daniel~S. Silver and Susan~G. Williams.
\newblock Mahler measure, links and homology growth.
\newblock {\em Topology}, 41(5):979--991, 2002.

\bibitem[SX91]{SarXue}
Peter Sarnak and Xiao~Xi Xue.
\newblock Bounds for multiplicities of automorphic representations.
\newblock {\em Duke Math. J.}, 64(1):207--227, 1991.

\bibitem[Wan72]{Wan}
Hsien~Chung Wang.
\newblock Topics on totally discontinuous groups.
\newblock pages 459--487. Pure and Appl. Math., Vol. 8, 1972.

\bibitem[Wor99]{Wor}
N.~C. Wormald.
\newblock Models of random regular graphs.
\newblock In {\em Surveys in combinatorics, 1999 ({C}anterbury)}, volume 267 of
  {\em London Math. Soc. Lecture Note Ser.}, pages 239--298. Cambridge Univ.
  Press, Cambridge, 1999.

\bibitem[You05]{You}
Robert Young.
\newblock Counting hyperbolic manifolds with bounded diameter.
\newblock {\em Geom. Dedicata}, 116:61--65, 2005.

\end{thebibliography}

\end{document}